\documentclass[12pt]{amsart}
\usepackage{amsmath,amssymb,amsthm}
\usepackage{needspace}
\usepackage[pdftex,
  pdftitle={Proofs of two conjectures of Berkovich and Dhar on hook lengths},
  pdfsubject={MATHEMATICS},
  pdfkeywords={partition inequalities, hook lengths, Gaussian polynomials},
  pdfproducer={LaTeX with hyperref},
  pdfcreator={pdflatex}]{hyperref}
\newtheorem{theorem}{Theorem}[section]
\newtheorem{corollary}[theorem]{Corollary}
\newtheorem{conjecture}[theorem]{Conjecture}
\newcommand{\qbinom}[3]{\genfrac{[}{]}{0pt}{}{#1}{#2}_{#3}}
\numberwithin{equation}{section}
\allowdisplaybreaks

\begin{document}
\title[Two conjectures of Berkovich and Dhar]
{Proofs of two conjectures of Berkovich and Dhar on hook lengths}
\author{Rong Chen}
\address{Department of Mathematics, Shanghai Normal University, People's Republic of China}
\email{rchen@shnu.edu.cn}
\date{\today}

\begin{abstract}
We prove stronger forms of two conjectures of Berkovich and Dhar
concerning a coefficientwise inequality and pairs of hooks of length
two in bounded partitions.  The coefficientwise inequality implies
the corresponding inequality for hook pairs.  In particular, for each
fixed size and common bound on the largest part, the total number
of such pairs in odd partitions is at least that in distinct
partitions.  We give combinatorial proofs by explicit weight-preserving
injections based on Glaisher's bijection and analytic proofs using
finite product identities and known hook generating functions.
\end{abstract}
\maketitle

\section{Introduction and main results}

The hook length of a cell in the Young diagram of a partition is one
plus the number of cells to its right in the same row and below it
in the same column.  We consider unordered pairs of distinct cells,
both of hook length two.  Thus a partition with $m$ hooks of length
two contributes $\binom{m}{2}$ pairs.

Ballantine, Burson, Craig, Folsom, and Wen
\cite[Proposition~3.3]{BBCFW} proved that
the total number of hooks of length two in all odd partitions of $n$
is at least the corresponding number in all distinct partitions of $n$.
Berkovich and Dhar \cite[Theorem~2.3]{BD} obtained a finite analog
with odd parts at most $2L-1$ and distinct parts at most $L$.
They also obtained generating functions for pairs of hooks of length
two \cite[Theorems~2.4 and~2.5]{BD}.  For these pairs, they proved
the inequality without bounds on the largest parts
\cite[Theorem~2.9]{BD} and conjectured a finite analog
\cite[Conjecture~2.6]{BD}.

In their concluding remarks, Berkovich and Dhar stated a
coefficientwise inequality \cite[Conjecture~4.1]{BD} which implies
their conjecture for pairs of hooks, and asked for an injective
proof.  We prove stronger forms of both conjectures.  In the
inequality for pairs of hooks, the bound on the distinct parts
is increased from $L$ to $2L$, while the bound on the odd parts
remains $2L-1$.  We give combinatorial proofs based on Glaisher's
bijection and analytic proofs using finite product identities
and the generating functions of Berkovich and Dhar.

For $L\geq1$, let $a_2(L,m,n)$ denote the number of partitions of $n$ into odd parts at most $2L-1$
having $m$ hooks of length two, and let $b_2(L,m,n)$ denote the number of
partitions of $n$ into distinct parts at most $L$ having
$m$ such hooks.  We write $|\pi|$ for the sum of the parts of a
partition $\pi$.  The $q$-Pochhammer symbol and the $q$-binomial
(Gaussian) coefficient are defined by
$$
 (a;q)_L=\prod_{j=0}^{L-1}(1-aq^j),\qquad
 \qbinom{m}{n}{q}=\frac{(q;q)_m}{(q;q)_n(q;q)_{m-n}}
 \quad(m\geq n\geq0),
$$
with $(a;q)_0=1$ and the Gaussian polynomial equal to zero
otherwise.  Empty sums are zero and empty products are one.
For formal power series, $A(q)\succeq B(q)$ means
that every coefficient of $A(q)-B(q)$ is nonnegative.

We first state the coefficientwise conjecture of Berkovich and Dhar.
We take $L\geq4$, so that all finite products have nonnegative lengths.

\Needspace{12\baselineskip}
\begin{conjecture}[{\cite[Conjecture~4.1]{BD}}]\label{c1}
For every integer $L\geq4$,
\begin{equation}\label{e1}
 \frac{1}{(q;q^2)_L}
 \left(\frac{1-q^{2L-2}}{1-q^2}
             +q^3\qbinom{L-2}{2}{q^2}\right)
 \succeq q\qbinom{L-2}{2}{q}(-q^3;q)_{L-4}.
\end{equation}
\end{conjecture}

The following theorem strengthens this conjecture.

\begin{theorem}\label{t2}
For every integer $L\geq2$,
\begin{equation}\label{e2}
 \frac{1}{(q;q^2)_{L-1}}
 \left(\frac{1-q^{2L-2}}{1-q^2}
             +q^3\qbinom{L-1}{2}{q^2}\right)
 \succeq q\qbinom{2L-2}{2}{q}(-q^3;q)_{2L-4}.
\end{equation}
\end{theorem}

\begin{corollary}\label{c3}
Conjecture~\ref{c1} holds.
\end{corollary}

Berkovich and Dhar also conjectured the following inequality for
the weighted count of hooks of length two.

\begin{conjecture}[{\cite[Conjecture~2.6]{BD}}]\label{c4}
For all integers $L\geq1$ and $n\geq0$,
\begin{equation}
 \sum_{m\geq0}\binom{m}{2}a_2(L,m,n)
 \ \geq\
 \sum_{m\geq0}\binom{m}{2}b_2(L,m,n).
\end{equation}
\end{conjecture}

We prove the following stronger inequality.

\begin{theorem}\label{t5}
For all integers $L\geq1$ and $n\geq0$,
\begin{equation}\label{e4}
 \sum_{m\geq0}\binom{m}{2}a_2(L,m,n)
 \ \geq\
 \sum_{m\geq0}\binom{m}{2}b_2(2L,m,n).
\end{equation}
\end{theorem}

Theorem~\ref{t5} follows from Theorem~\ref{t2}
by comparing the generating functions for the two weighted counts.
This is analogous to the implication between the two conjectures
noted in \cite[Section~4]{BD}.  The combinatorial proof uses the
injection of Theorem~\ref{t2} to construct an injection
between partitions with marked hook pairs.

\begin{corollary}\label{c6}
Conjecture~\ref{c4} holds.
\end{corollary}

Corollary~\ref{c6} follows from Theorem~\ref{t5}, since
$b_2(L,m,n)\leq b_2(2L,m,n)$ for every $m,n$.
For a common bound on the largest part,
Theorem~\ref{t5} gives, for every $L\geq1$,
$$
 \sum_{m\geq0}\binom{m}{2}
 a_2\!\left(\left\lceil\frac L2\right\rceil,m,n\right)
 \geq\sum_{m\geq0}\binom{m}{2}b_2(L,m,n).
$$
Indeed, the largest permitted odd part is $L-1$ when $L$ is even
and $L$ when $L$ is odd; in the latter case we restrict the distinct
bound from $L+1$ to $L$.  Thus both sides count pairs of hooks of
length two in partitions with largest part at most $L$.

The rest of the paper is organized as follows.  We give
combinatorial proofs in Section~\ref{s2} and
analytic proofs in Section~\ref{s3}.

\section{Combinatorial proofs}\label{s2}

\begin{proof}[Proof of Theorem~\ref{t2}]
We interpret both sides of \eqref{e2} combinatorially.  The standard Gaussian interpretation
(see example \cite[Section~2.2.5]{Pak}) gives
$$
 \qbinom{2L-2}{2}{q}
   =\sum_{0\leq a\leq b\leq2L-4}q^{a+b}.
$$
Thus the coefficient of $q^n$ on the right counts triples
$(\mu,a,b)$, where $\mu$ has distinct parts in
$\{3,\ldots,2L-2\}$, $0\leq a\leq b\leq2L-4$, and
\begin{equation}\label{wt}
 |\mu|+1+a+b=n.
\end{equation}
On the left, $\pi$ is an odd partition with largest part at most
$2L-3$.  The coefficient of $q^n$ counts the disjoint union of
\begin{enumerate}
\item pairs $(\pi,i)$ with $0\leq i\leq L-2$ and $|\pi|+2i=n$;
\item triples $(\pi,c,d)$ with $0\leq c\leq d\leq L-3$ and
      $|\pi|+3+2c+2d=n$.
\end{enumerate}
The sums over the integer parameters are
$$
 \sum_{i=0}^{L-2}q^{2i}=\frac{1-q^{2L-2}}{1-q^2},\qquad
 \sum_{0\leq c\leq d\leq L-3}q^{3+2c+2d}
       =q^3\qbinom{L-1}{2}{q^2}.
$$
The second class is empty for $L=2$.

Let $g$ denote Glaisher's bijection from distinct partitions to odd
partitions \cite[Section~3.2.1]{Pak}.  It replaces each part $2^su$,
with $u$ odd, by $2^s$ copies of $u$, and preserves the sum of the
parts.  Here $\mu\cup(1^e,2^f)$ denotes the partition obtained by
adjoining $e$ ones and $f$ twos to $\mu$.
We define $\varphi$ by the following three cases.

\emph{Case I: $a=0$ and $b=2i$.}
$$
 \varphi(\mu,0,2i)=\bigl(g(\mu\cup(1)),i\bigr).
$$
Here $0\leq i\leq L-2$, and the weight of the image is
$|\mu|+1+2i$.

\emph{Case II: $a=0$ and $b=2i+1$.}
$$
 \varphi(\mu,0,2i+1)=\bigl(g(\mu\cup(2)),i\bigr).
$$
Here $0\leq i\leq L-3$, and the weight of the image is
$|\mu|+2+2i$.

\emph{Case III: $a>0$.}  There are unique integers $c,d$ and
$e,f\in\{0,1\}$ such that
\begin{equation}\label{par}
 a=2c+1+f,\qquad b=2d+1+f+e.
\end{equation}
The parity of $a-1$ determines $f$, and the parity of $b-a$
determines $e$.  The inequalities $a\leq b\leq2L-4$ give
$0\leq c\leq d\leq L-3$.  In this case,
$$
 \varphi(\mu,a,b)=\bigl(g(\mu\cup(1^e,2^f)),c,d\bigr).
$$
Its weight is
$$
 |\mu|+e+2f+3+2c+2d=|\mu|+1+a+b.
$$

In each case, the partition to which $g$ is applied has distinct
parts at most $2L-2$.  Its image therefore has odd parts at most
$2L-3$.  Thus $\varphi$ maps the triples counted on the right of
\eqref{e2} to the stated classes on the left, preserving
weight.

We now show that $\varphi$ is injective.  Since $g$ is one-to-one
and $\mu$ has no parts $1$ or $2$, an image in the first class
uniquely determines $\mu$ and whether Case~I or Case~II applies;
the integer $i$ then determines $a,b$.  An image in the
second class uniquely determines $\mu,e,f$, and its integers
$c,d$ determine $a,b$ by \eqref{par}.  Hence distinct triples
have distinct images, proving \eqref{e2}.
\end{proof}

\begin{proof}[Proof of Corollary~\ref{c3}]
Let $L\geq4$.  The right side of \eqref{e1} counts
triples $(\mu,a,b)$ of weight $|\mu|+1+a+b$, where $\mu$ has
distinct parts in $\{3,\ldots,L-2\}$ and
$0\leq a\leq b\leq L-4$.  We apply the map $\varphi$ defined in the proof of
Theorem~\ref{t2} to these triples, with $L$
replaced by $L-1$.  Its images have odd parts at most $2L-5$, together with
an integer $0\leq i\leq L-3$ or two integers
$0\leq c\leq d\leq L-4$.

The left side of \eqref{e1} counts pairs $(\pi,i)$
and triples $(\pi,c,d)$ with the same respective weights, where
$\pi$ has odd parts at most $2L-1$ and
$$
 0\leq i\leq L-2,\qquad 0\leq c\leq d\leq L-4.
$$
These classes contain all the preceding images.  The restricted
map is therefore a weight-preserving injection proving
\eqref{e1}.
\end{proof}

\begin{proof}[Proof of Theorem~\ref{t5}]
Let $\mathcal O$ be the set of pairs $(\pi,S)$ in which $\pi$
is a partition of $n$ into odd parts at most $2L-1$ and $S$ is
a two-element set of cells of hook length two in $\pi$.
Define $\mathcal D$ in the same way, with $\pi$ a partition of
$n$ into distinct parts at most $2L$.  Then
$$
 \#\mathcal O=\sum_{m\geq0}\binom{m}{2}a_2(L,m,n),\qquad
 \#\mathcal D=\sum_{m\geq0}\binom{m}{2}b_2(2L,m,n).
$$
For $L=1$, both sets are empty.  Suppose that $L\geq2$.

Let $\mathcal T$ consist of triples $(\mu,a,b)$, where $\mu$ is
a partition into distinct parts in $\{3,\ldots,2L-2\}$ and
$a,b$ are integers satisfying
$$
 0\leq a\leq b\leq2L-4,\qquad |\mu|+1+a+b=n-5.
$$
Let $\mathcal U$ be the disjoint union of the two classes
$$
 \begin{array}{ll}
 (\pi,k),&0\leq k\leq L-2,\quad |\pi|+2k=n-5,\\[2pt]
 (\pi,c,d),&0\leq c\leq d\leq L-3,\quad
                 |\pi|+3+2c+2d=n-5,
 \end{array}
$$
where $\pi$ is a partition into odd parts at most $2L-3$ and
$k,c,d$ are integers.  The required injection will be the composition
$$
 \mathcal D\xrightarrow{\alpha}\mathcal T
 \xrightarrow{\varphi}\mathcal U\xrightarrow{\psi}\mathcal O.
$$

For the first map, $\forall(\pi_d,S)\in\mathcal D$, write
$\pi_d=(\lambda_1,\ldots,\lambda_r)$, with
$2L\geq\lambda_1>\cdots>\lambda_r\geq1$.
By \cite[Section~3.2]{BD}, every hook of length two in $\pi_d$
is horizontal and is uniquely determined by its row.
Let $i<j$ be the row indices of the two cells in $S$, counted
from top to bottom.  With $\lambda_{r+1}=0$, we have
$$
 \lambda_i-\lambda_{i+1}\geq2,\qquad
 \lambda_j-\lambda_{j+1}\geq2.
$$
Define $\delta=(\delta_1,\ldots,\delta_r)$ by
$$
 \delta_h=
 \begin{cases}
  \lambda_h-2,&1\leq h\leq i,\\
  \lambda_h-1,&i<h\leq j,\\
  \lambda_h,&j<h\leq r.
 \end{cases}
$$
The gap conditions give
$$
 2L-2\geq\delta_1>\cdots>\delta_r\geq1,\qquad
 |\delta|=n-i-j.
$$
Let $\mu$ be the sequence obtained from $\delta$ by deleting
$\delta_{r-j+1}$ and $\delta_{r-i+1}$, keeping the parts above
both deleted parts unchanged, and increasing those between
them by $1$ and those below both by $2$.
We define $\alpha(\pi_d,S)=(\mu,a,b)$ by
\begin{align}
 a&=\delta_{r-i+1}-1,\qquad
 b=\delta_{r-j+1}-2,\label{map}\\*
 \mu&=\bigl(\delta_1,\ldots,\delta_{r-j},
       \delta_{r-j+2}+1,\ldots,\delta_{r-i}+1,\delta_{r-i+2}+2,\ldots,\delta_r+2\bigr),\notag
\end{align}
where empty blocks are omitted.
Since $1\leq\delta_{r-i+1}<\delta_{r-j+1}\leq2L-2$,
we have $0\leq a\leq b\leq2L-4$.
The three blocks are strictly decreasing, with parts in
$[b+3,2L-2]$, $[a+3,b+2]$, and $[3,a+2]$, respectively.
These intervals are disjoint and occur in decreasing order.
Thus $\mu$ is a partition into distinct parts
in $\{3,\ldots,2L-2\}$.  Moreover,
\begin{align}
 |\mu|&=|\delta|-(a+1)-(b+2)+(j-i-1)+2(i-1),\label{enc}\\
 |\mu|+1+a+b&=|\delta|+i+j-5=n-5.\notag
\end{align}
Hence $\alpha$ maps $\mathcal D$ into $\mathcal T$.
If two elements of $\mathcal D$ have the same image $(\mu,a,b)$,
then their omitted parts of $\delta$ are $a+1$ and $b+2$.
The three disjoint intervals determine the block in \eqref{map}
containing each part of $\mu$ and hence its corresponding part
of $\delta$.  Thus their partitions $\delta$ agree.
The positions of $a+1$ and $b+2$ are $r-i+1$ and $r-j+1$,
respectively, so $i,j$ also agree.  The definition of $\delta$
then determines $\pi_d$ and the two cells of $S$.
Therefore $\alpha$ is injective.

The second map $\varphi:\mathcal T\to\mathcal U$ is precisely
the injection constructed in the proof of Theorem~\ref{t2},
with the same $L$ and weight $n-5$ in \eqref{wt}.

It remains to define an injection $\psi:\mathcal U\to\mathcal O$.
For an odd partition $\rho$, let $H_s(\rho)$ denote the penultimate
cell in its bottommost row of length $s>1$, and let $V_1(\rho)$
denote the cell in its second-bottommost row of length one, whenever
these cells exist.  Their hooks are horizontal and vertical,
respectively, and have length two \cite[Section~3.1]{BD}.
For the pairs and triples in $\mathcal U$, define
$$
\begin{aligned}
 \psi(\pi,k)&=\bigl(\rho,\{V_1(\rho),H_{2k+3}(\rho)\}\bigr),
 &\rho&=\pi\cup(1^2,2k+3),\\
 \psi(\pi,c,d)&=\bigl(\rho,\{H_{2c+3}(\rho),H_{2d+5}(\rho)\}\bigr),
 &\rho&=\pi\cup(2c+3,2d+5).
\end{aligned}
$$
The adjoined parts ensure that the indicated cells exist.
The bounds defining $\mathcal U$ give
$$
 3\leq2k+3\leq2L-1,\qquad
 3\leq2c+3<2d+5\leq2L-1,
$$
so $\rho$ has odd parts at most $2L-1$ and two distinct marked
cells of hook length two.  In the two cases, respectively,
$$
 |\rho|=|\pi|+2+(2k+3)=n,\qquad
 |\rho|=|\pi|+(2c+3)+(2d+5)=n.
$$
Thus $\psi$ maps $\mathcal U$ into $\mathcal O$.
The two classes have disjoint images: the first has a vertical
and a horizontal marked hook, whereas the second has two
horizontal marked hooks.  Within either class, the marked row
lengths determine $k$ or $c,d$.  Equal images therefore have
the same adjoined parts and the same $\pi$, proving that
$\psi$ is injective.

The composition $\psi\circ\varphi\circ\alpha$ is an injection
from $\mathcal D$ to $\mathcal O$.  Hence
$\#\mathcal D\leq\#\mathcal O$, proving \eqref{e4}.
\end{proof}

\section{Analytic proofs}\label{s3}

Gaussian polynomials have nonnegative coefficients and are
coefficientwise nondecreasing in their upper index.  These properties
follow from their initial values and the standard recurrence obtained
from the finite $q$-binomial
theorem \cite[Equation~(3.3.6)]{Andrews}
\begin{equation}\label{rec}
 \qbinom{m}{n}{q}
 =\qbinom{m-1}{n}{q}
     +q^{m-n}\qbinom{m-1}{n-1}{q}
 \qquad(1\leq n\leq m).
\end{equation}
All products used to multiply an inequality below have
nonnegative coefficients.

\begin{proof}[Proof of Theorem~\ref{t2}]
For $L>2$, substitution of the product formulas for the Gaussian
polynomials gives
\begin{align}
 &(1+q)(1+q^2)
 \left(\frac{1-q^{2L-2}}{1-q^2}
             +q^3\qbinom{L-1}{2}{q^2}\right)
       -q\qbinom{2L-2}{2}{q} \label{pos}\\
 &=\frac{1-q^{2L-2}}{1-q^2}
 \left((1+q)(1+q^2)
 +\frac{q^3(1-q^{2L-4})-q(1-q^{2L-3})}{1-q}\right)\notag\\
 &=\frac{1-q^{2L-2}}{1-q^2}(1+q^3+q^{2L-2})\succeq0.\notag
\end{align}
At $L=2$, both $\qbinom{L-1}{2}{q^2}$ and $1-q^{2L-4}$ vanish,
so the identity remains valid.
Cancellation in the finite products also gives
\begin{equation}\label{prod}
 \frac1{(q;q^2)_{L-1}}
 =\frac{(-q;q)_{2L-2}}{(q^{2L};q^2)_{L-1}}
 \succeq(-q;q)_{2L-2},
\end{equation}
since $(q^{2L};q^2)_{L-1}^{-1}$ has nonnegative coefficients
and constant term one.
Using \eqref{pos}, \eqref{prod}, and
$(-q;q)_{2L-2}=(1+q)(1+q^2)(-q^3;q)_{2L-4}$, we obtain
\begin{align}
 &\frac1{(q;q^2)_{L-1}}
 \left(\frac{1-q^{2L-2}}{1-q^2}
              +q^3\qbinom{L-1}{2}{q^2}\right)\\
 &\quad\succeq(-q;q)_{2L-2}
 \left(\frac{1-q^{2L-2}}{1-q^2}
              +q^3\qbinom{L-1}{2}{q^2}\right)\notag\\
 &\quad\succeq q\qbinom{2L-2}{2}{q}(-q^3;q)_{2L-4}.\notag
\end{align}
This proves \eqref{e2}.
\end{proof}

\begin{proof}[Proof of Corollary~\ref{c3}]
Let $L\geq4$.  Applying Theorem~\ref{t2} with
$L$ replaced by $L-1$, we obtain
\begin{align*}
 &\frac1{(q;q^2)_L}
 \left(\frac{1-q^{2L-2}}{1-q^2}
                +q^3\qbinom{L-2}{2}{q^2}\right)\\
 &\quad\succeq\frac1{(q;q^2)_{L-2}}
 \left(\frac{1-q^{2L-4}}{1-q^2}
                +q^3\qbinom{L-2}{2}{q^2}\right)\\
 &\quad\succeq q\qbinom{2L-4}{2}{q}(-q^3;q)_{2L-6}\\
 &\quad\succeq q\qbinom{L-2}{2}{q}(-q^3;q)_{L-4}.
\end{align*}
For the first inequality, we use
$$
 \frac{1-q^{2L-2}}{1-q^2}
 =\frac{1-q^{2L-4}}{1-q^2}+q^{2L-4},\qquad
 \frac{(q;q^2)_{L-2}}{(q;q^2)_L}
 =\frac1{(1-q^{2L-3})(1-q^{2L-1})}\succeq1.
$$
For the last, \eqref{rec} gives
$\qbinom{2L-4}{2}{q}\succeq\qbinom{L-2}{2}{q}$, and
$$
 \frac{(-q^3;q)_{2L-6}}{(-q^3;q)_{L-4}}
 =(-q^{L-1};q)_{L-2}\succeq1.
$$
\end{proof}

\begin{proof}[Proof of Theorem~\ref{t5}]
For $L=1$, both sides of \eqref{e4} vanish.  Let $L\geq2$.
We use \cite[Theorem~2.4]{BD}, with a minor typographical correction
in their Equation~(2.3): the last $q^{2L}$ in the mixed term should read
$q^{2L-2}$.  Indeed, substituting $z=1+t$ in
\cite[Equation~(3.1)]{BD} and extracting the coefficient of $t^2$
gives the mixed term below, whose equalities follow by geometric
summation:
\begin{align*}
 &\sum_{2\leq i<j\leq L}
       \bigl(q^{2i+4j-3}+q^{4i+2j-3}\bigr)\\
 &\quad=q^9\left(
   \frac{(1-q^{2L-2})(1-q^{4L-4})}{(1-q^2)(1-q^4)}
       -\frac{1-q^{6L-6}}{1-q^6}\right)\\
 &\quad=q^{11}\frac{1-q^2}{1-q^{2L}}\qbinom{L}{3}{q^2}
 \bigl((1+q^2)(1+q^{2L-2})+q^2+q^{2L-2}\bigr)\succeq0.
\end{align*}
This gives the nonnegativity asserted in \cite[Remark~3]{BD}
with the same correction.  Keeping the first and third terms in
the corrected version of \cite[Equation~(2.3)]{BD}, we obtain
\begin{equation}\label{odd}
 \sum_{n,m\geq0}\binom{m}{2}a_2(L,m,n)q^n
 \succeq\frac{q^5}{(q;q^2)_L}
 \left(\frac{1-q^{2L-2}}{1-q^2}
                +q^3\qbinom{L-1}{2}{q^2}\right).
\end{equation}
On the distinct side, \cite[Theorem~2.5]{BD}, with $L$ replaced
by $2L$, gives
\begin{equation}\label{dist}
 \sum_{n,m\geq0}\binom{m}{2}b_2(2L,m,n)q^n
   =q^6\qbinom{2L-2}{2}{q}(-q^3;q)_{2L-4}.
\end{equation}
Since $(q;q^2)_L^{-1}\succeq(q;q^2)_{L-1}^{-1}$,
\eqref{odd} and Theorem~\ref{t2} yield
\begin{align*}
 \sum_{n,m\geq0}\binom{m}{2}a_2(L,m,n)q^n
 &\succeq\frac{q^5}{(q;q^2)_{L-1}}
 \left(\frac{1-q^{2L-2}}{1-q^2}
                +q^3\qbinom{L-1}{2}{q^2}\right)\\
 &\succeq q^6\qbinom{2L-2}{2}{q}(-q^3;q)_{2L-4}.
\end{align*}
Together with \eqref{dist}, this proves \eqref{e4}
by comparison of coefficients.
\end{proof}

\end{document}